\documentclass[11pt,CJK]{article}
\usepackage[numbers,sort&compress]{natbib}
\usepackage{enumerate}
\usepackage{amscd}
\usepackage{array}
\usepackage{amsmath}
\usepackage{latexsym}
\usepackage{amsfonts}
\usepackage{amssymb}
\usepackage{amsthm}
\usepackage{verbatim}
\usepackage{mathrsfs}
\usepackage{enumerate}
\usepackage{hyperref}

\theoremstyle{plain}\newtheorem{definition}{Definition}[section]
\theoremstyle{definition}\newtheorem{theorem}{Theorem}[section]
\theoremstyle{plain}\newtheorem{lemma}[theorem]{Lemma}
\theoremstyle{plain}\newtheorem{coro}[theorem]{Corollary}
\theoremstyle{plain}
\theoremstyle{remark}\newtheorem{remark}{Remark}[section]
\theoremstyle{definition}
\theoremstyle{plain}
\usepackage{xcolor}
\newcommand{\wred}[1]{\textcolor{black}{#1}}
\newcommand{\hs}[1]{\textcolor{red}{#1}}

\newcommand{\norm}[1]{\left\|#1\right\|}
\newcommand{\Div}{\mathrm{div}\,}
\newcommand{\B}{\Big}

\newcommand{\be}{\begin{equation}}
\newcommand{\ee}{\end{equation}}
 \newcommand{\ba}{\begin{aligned}}
 \newcommand{\ea}{\end{aligned}}

\providecommand{\bysame}{\leavevmode\hbox to3em{\hrulefill}\thinspace}
  \newcommand{\f}{\frac}
    
  \newcommand{\ben}{\begin{enumerate}}
   \newcommand{\een}{\end{enumerate}}

\newcommand{\ti}{\nabla}

\newcommand{\Rmnum}[1]{\expandafter\@slowromancap\romannumeral #1@}

\allowdisplaybreaks

\numberwithin{equation}{section}
\begin{document}%\begin{CJK*}{GBK}{fs}
%%%%%%%%%%%%%%%%%%%%%%%%%%%%%%%%%%%%%%%%%%%%%%%%%%%%%%%%%%%%%%%%%%%%%%%%%%%%%%%%%%%%%%%%%%%%%%%%%%%%
\title{%An elementary proof of  interior regularity criteria involving two  components in the 3D Navier-Stokes equations
 $\varepsilon$-regularity criteria in Lorentz spaces   to the 3D Navier-Stokes equations }
\author{ Yanqing Wang\footnote{ Department of Mathematics and Information Science, Zhengzhou University of Light Industry, Zhengzhou, Henan  450002,  P. R. China Email: wangyanqing20056@gmail.com},\;~ Wei Wei\footnote{Center for Nonlinear Studies, School of Mathematics, Northwest University, Xi'an, Shaanxi 710127, China Email: ww5998198@126.com } ~  and\, Huan Yu \footnote{ School of Applied Science, Beijing Information Science and Technology University, Beijing, 100192, P. R. China Email:  yuhuandreamer@163.com}
 }
\date{}
\maketitle
\begin{abstract}
 In  this paper,  we  are concerned with regularity of suitable weak solutions of the 3D Navier-Stokes equations in Lorentz spaces. We obtain $\varepsilon$-regularity criteria in terms of either the velocity, the gradient of the velocity,  the vorticity,  or deformation tensor   in Lorentz spaces. As an application, this allows us to extend  the result involving Leray's blow up rate in time, and to show that the number of singular points of weak solutions belonging to  $ L^{p,\infty}(-1,0;L^{q,l}(\mathbb{R}^{3})) $ and $ {2}/{p}+{3}/{q}=1$ with $3<q<\infty$ and $q\leq l <\infty$ is finite.
 \end{abstract}
%%%%%%%%%%
\section{Introduction}
\label{intro}
\setcounter{section}{1}\setcounter{equation}{0}
We focus our attention on the 3D  Navier-Stokes system
\be\left\{\ba\label{NS}
&u_{t} -\Delta  u+ u\cdot\ti
u  +\nabla \Pi=0, \\
&\Div u=0,\\
&u|_{t=0}=u_0,
\ea\right.\ee
 where \wred{the unknown vector $u=u(x,t)$ describes the flow  velocity field}, the scalar function $\Pi$ represents the   pressure.
 The  initial datum $u_{0}$ is given and satisfies the divergence-free condition.

The global   weak  solutions to   the 3D  Navier-Stokes equations  was constructed by Leray  \cite{[Leray1]}  for
the Cauchy problem and by Hopf \cite{[Hopf]} for the Dirichlet problem. However, whether such a weak solution is unique or regular is still an outstanding open problem.
  In
a pioneering   work  \cite{[Serrin]}, Serrin   presented a regularity criteria
 to  the  weak solutions to the 3D Navier-Stokes equations, namely,  the
 Leray-Hopf weak  solutions $u$  is bounded in    some neighbourhood of point $(0,0)$  if   $u$ satisfies
\be\label{serrin}
u\in  L^{p} (-1,0;L^{q}( B(1))) ~~~ \text{with}~~~~2/p+3/q<1, ~~q>3.
\ee
Here and in what follows, $B(\rho)$ denotes the ball of center $0$ and radius $\rho$. The critical case $2/p+3/q=1$ in \eqref{serrin} was handled by
Struwe  in \cite{[Struwe]}. The limiting case  $L^{\infty} (-1,0;L^{3}( B(1)))$ was solved by Escauriaza,  Seregin and \v{S}ver\'{a}k  in \cite{[ESS]}.
Note that  the  norm  $\|\cdot\|_{L_t^{p}L_x^{q}}$ with $2/p+3/q=1$ is scaling invariant for $u$ under the natural scaling of \eqref{NS}.

It is well-known that
Lorentz spaces $L^{r,s}(\Omega)$ $(s\geq r)$ are larger than the Lebesgues
spaces  $L^{r}(\Omega)$. A natural question arises whether results similar to  \eqref{serrin} hold in Lorentz spaces. Indeed,
Takahashi \cite{[Takahashi]}  gave the first affirmative answer  and improved \eqref{serrin} to allow the time direction to belong  to Lorentz spaces, more precisely,
 \be\label{Takahashi}
u\in L^{p,\infty} (-1,0;L^{q}( B(1))) ~~~ \text{with}~~~~2/p+3/q=1, ~~3<q<\infty.
\ee
Afterwards, it is shown that
an alternative assumption
of \eqref{Takahashi} is
 \be\label{CP}
u \in L^{p,\infty} (-1,0;L^{q,\infty}( B(1)))  ~~~ \text{with}~~~~2/p+3/q=1, ~~ 3<q<\infty,
\ee which is due to Chen and
Price \cite{[CP]}  and Sohr \cite{[Sohr]}.
The endpoint case $q=3$ in \eqref{CP} was considered by
 Kozono  and his coauthor Kim in \cite{[Kozono],[KK]}.
%5Wang and Zhang and Phuc
%\be\label{WZ}u \in L^{ \infty} (-1,0;L^{3,l}( \mathbb{R}^{3})), 3<l<\infty.
%\ee
 % Choe,  Wolf,  Yang
 % \be\label{serrin}
%u\in  L^{\infty} (-1,0;L^{3,\infty}( B(1))) ~~~ \text{and}~~~~\f{1}{r^{3}}\mu\{x\in B(x,r):|u|>\f{\varepsilon}{r}\}\leq\varepsilon.
%\ee

In an interesting work,  Gustafson, Kang and Tsai \cite{[GKT]} found that all regularity criteria \eqref{serrin}-\eqref{CP} and the results in \cite{[Kozono],[KK],[Struwe]}  can be derived from the following $\varepsilon$-regularity criteria of suitable weak solutions to   system \eqref{NS}:  if  there exists an absolute constant $\varepsilon$ such that
\begin{align}\label{tsai1}
&\limsup_{\varrho\to 0 }\,\, \varrho^{1- \frac 2p- \frac 3 q}
\|u-\overline{u}_{\varrho}\|_{L^{p}(-\varrho^{2},0;L^{q}(B(\varrho))) } \leq \varepsilon,\quad
1\leq 2/p +3/q\leq 2,\;  1\leq p, q \leq \infty,
 \end{align}
 then $u$  is bounded in    some neighbourhood of point $(0,0)$.
Here, suitable weak solutions of  the  3D Navier-Stokes equations  were  introduced by Scheffer \cite{[Scheffer1],[Scheffer2]} and Caffarelli,
Kohn and Nirenberg \cite{[CKN]} to estimate
 the size of the singular points in the Navier-Stokes system.
 A point $ (x,t)$   is said to be regular if $u$ belongs to $ L^{\infty}$ in a neighborhood of $(x,t)$.
Otherwise, it is   singular. Denote the possible singular points set in system \eqref{NS} by $\mathcal{S}.  $
 Two  kinds  of  $\varepsilon$-regularity criteria to the suitable weak solutions of \eqref{NS} were derived in \cite{[CKN]}: $(0,0)$ is a regular point provided that one of the two following conditions holds, for an absolute positive constant $\varepsilon$,
\begin{equation}
\label{ckn1}		
\norm{u}_{L^{3}(-1,0;L^{3}(B(1))} +\|u\Pi\|_{L^{1}(-1,0;L^{1}(B(1))} +\|\Pi\|_{L^{\f54}(-1,0;L^{1}(B(1))}\leq \varepsilon.
\end{equation}
and
\begin{equation}
\label{ckn2}
\limsup_{\varrho\rightarrow0}  \varrho^{-1}\norm{\nabla u}_{ L^{2}(-\varrho^{2},0;L^{2}(B(\varrho))) }^{2} \leq \varepsilon.
\end{equation}
The difference between  \eqref{ckn1} and \eqref{ckn2} is that the    former requires only a radius (one scale) and the   latter needs infinite radii (finitely many scales).
Since then, there have been
extensive mathematical investigations
 of regularity
of suitable weak solutions and many   regularity conditions are established (see
  \cite{[HWZ],[WW16],[GKT],[TX],[WWZ],[WZ]} and references therein).
However, almost all known results involving $\varepsilon$-regularity criteria are discussed in usual Lebesgues spaces and
there is a little literature for investigating $\varepsilon$-regularity criteria in Lorentz spaces. An objective of this paper is to study the regularity of suitable weak solutions in Lorentz spaces.

Before  formulating  our results, we
recall that
 the gradient of  the velocity field $\nabla u$ can be split   into a symmetric part $\mathcal{D}$ and an antisymmetric part $\Omega$, that is,
\be\ba\label{tidufenjie} \nabla u&=\f{1}{2}(\nabla u+\nabla u^{^{\text{T}}})+\f{1}{2}(\nabla u-\nabla u^{^{\text{T}}})\\
&=:\mathcal{D}(u)+\Omega(u).\ea\ee
  $\mathcal{D}(u)$ and  $\Omega(u)$ are usually called as the  deformation tensor or rate-of-strain tensor and  the rotation   tensor, respectively. (See eg. \cite{[MB]}).
Regularity criteria in terms of deformation tensor in fluid mechanics can be found in \cite{[HLX],[KT],[Ponce],[MNS],[WW16]}.
The first result can be stated as follows.
\begin{theorem}\label{the1.1}Let $(u,\,\Pi)$ be a suitable weak  solution  to \eqref{NS} in $Q(\varrho)$.
Then there exists a positive constant $\varepsilon_{1}$ such that $(0,0)$ is regular point provided that,
 one of    the following four conditions holds
\begin{enumerate}[(1)]
 \item \label{them1.1}$u \in L^{p}(-\varrho^{2},0; L ^{q,\infty}(B(\varrho))) $ and
    \be\label{u1} \limsup\limits_{\varrho\rightarrow0}
 \varrho^{-1 }
\|u-\overline{ u}_{\varrho}  \|_{L^{p}(-\varrho^{2},0; L ^{q,\infty}(B(\varrho)))} \leq\varepsilon_{1},
\ee
 where
\be\label{u11}2/p+3/q =2,~~~~~\text{with} ~~1\leq  p\leq\infty.
\ee
\be \label{u2}\text{or}~~ u \in L^{p,\infty}(-\varrho^{2},0; L ^{q,\infty}(B(\varrho)))~~\text{with}~~\limsup\limits_{\varrho\rightarrow0}
 \varrho^{1-\f{2}{p}-\f{3}{q} }
\|u-\overline{ u_{\varrho} }\|_{L^{p,\infty}(-\varrho^{2},0; L ^{q,\infty}(B(\varrho)))} \leq\varepsilon_{1},
\ee
 where
\be\label{u22}2/p+3/q <2,~~~~~\text{with}~~ 1<  p\leq\infty.
\ee
 \item  \label{them1.2}  $  \nabla u \in L^{p}(-\varrho^{2},0; L ^{q,\infty}(Q(\varrho)))$ and
     \be\label{tiduu1} ~~~\limsup\limits_{\varrho\rightarrow0}
 \varrho^{-1 }
\|\nabla u \|_{L^{p}(-\varrho^{2},0; L ^{q,\infty}(B(\varrho)))} \leq\varepsilon_{1},
\ee
 where
\be\label{tiduu11}2/p+3/q =3,~~~~~\text{with} ~~~ 1< p<\infty.
\ee
\be\label{tiduu2}\text{or}~~~~~~~~\nabla u \in L^{p,\infty}(-\varrho^{2},0; L ^{q,\infty}(B(\varrho)))~~~\text{with}~~~\limsup\limits_{\varrho\rightarrow0}
 \varrho^{2-\f{2}{p}-\f{3}{q} }
\|\nabla u \|_{L^{p,\infty}(-\varrho^{2},0; L ^{q,\infty}(B(\varrho)))} \leq\varepsilon_{1},
\ee
 where
\be\label{tiduu22}2/p+3/q <3,~~~\text{with}~~ 1< p<\infty.
\ee\item \label{3inth1.1}  $\text{curl\,} u  \in L^{p}(-\varrho^{2},0; L ^{q,\infty}(B(\varrho)))~~~\text{and}$
  \be\label{xduu1}~~~\limsup\limits_{\varrho\rightarrow0}
 \varrho^{-1 }
\|\text{curl\,} u  \|_{L^{p}(-\varrho^{2},0; L ^{q,\infty}(B(\varrho)))} \leq\varepsilon_{1},
\ee
 where
\be\label{xduu11}2/p+3/q =3,~~~~~\text{with} ~~~ 1< p<\infty.
\ee
\be\label{xduu2}\text{or}~~ \text{curl\,} u  \in L_{t}^{p,\infty}L ^{q,\infty}(B(\varrho))~  \text{with}~ \limsup\limits_{\varrho\rightarrow0}
 \varrho^{2-\f{2}{p}-\f{3}{q} }
\|\text{curl\,} u \|_{L^{p,\infty}(-\varrho^{2},0; L ^{q,\infty}(B(\varrho)))} \leq\varepsilon_{1},
\ee
 where
\be\label{xduu22}2/p+3/q <3,~~~\text{with}~~ 1< p<\infty.\ee
\item \label{4inth1.1}   $\mathcal{D}(u)  \in L^{p}(-\varrho^{2},0; L ^{q,\infty}(B(\varrho)))$  and \be\label{dz1} \limsup\limits_{\varrho\rightarrow0}
 \varrho^{-1 }
\|\mathcal{D}(u)  \|_{L^{p}(-\varrho^{2},0; L ^{q,\infty}(B(\varrho)))} \leq\varepsilon_{1},
\ee
 where
\be\label{dz11}2/p+3/q =3,~~~~~\text{with} ~~~ 1\leq p<\infty.
\ee
\be\label{dz2}\text{or}~~~~~~\mathcal{D}(u)  \in L^{p,\infty}(-\varrho^{2},0; L ^{q,\infty}(B(\varrho)))~~\text{with}~~\limsup\limits_{\varrho\rightarrow0}
 \varrho^{2-\f{2}{p}-\f{3}{q} }
\|\mathcal{D}(u) \|_{L^{p,\infty}(-\varrho^{2},0; L ^{q,\infty}(B(\varrho)))} \leq\varepsilon_{1},
\ee
 where
\be\label{dz22}2/p+3/q <3,~~~\text{with}~~ 1< p<\infty.\ee

\end{enumerate}
\end{theorem}
\begin{remark}
Theorem \ref{the1.1} is an improvement of \eqref{tsai1}, \eqref{ckn2} and corresponding result in \cite{[TX],[GKT],[WW16]}.
\end{remark}
We present some comments on   Theorem \ref{the1.1}.
Thanks to inclusion \eqref{Inclusion} on bounded domain in Lorentz spaces, we observe that \eqref{u2}, \eqref{tiduu2}, \eqref{xduu2}, \eqref{dz2} are just the straightforward  consequences of known results in \cite{[GKT],[WW16]}. Here, we give the proof of \eqref{u2}  and   leave others to the  interested  readers.
Indeed, for $(p,q)$ satisfying \eqref{u22}, there exist $\delta_{1},\delta_{2}>0$ such that
$$\f{2}{p-\delta_{1}}+\f{3}{q-\delta_{2}} <2.$$
Then, by means of inclusion \eqref{Inclusion}, we have
$$\|u\|_{L^{p-\delta_{1}}(-\varrho^{2},0;L^{q-\delta_{2}}(B(\varrho)))}
\leq C\varrho^{\f{2\delta_{1}}{p(p-\delta_{1})}+\f{3\delta_{2}}{q(q-\delta_{2})}}\|u\|_{L^{p,\infty}(-\varrho^{2},0;L^{q,\infty}(B(\varrho)))}.
$$
This together with \eqref{tsai1} implies the  desired  result.

The  more interesting   results in Theorem \ref{the1.1} are \eqref{u1}, \eqref{tiduu1}, \eqref{xduu1}, \eqref{dz1}.
By means of  Poincar\'e-Sobolev inequality \eqref{pil} in Lorentz spaces, we can prove \eqref{them1.1} \eqref{them1.2} in Theorem \ref{the1.1}.
To the knowledge of the authors, this type  of inequality and its proof were first   mentioned  by Mal\'y in
 \cite[Remark 8.3, p.15]{[Maly]}. The proof relies   on   the concept of a median of function there. As said by Mal\'y in \cite{[Maly]},  one can follow the path of the proof of classical Poincar\'e-Sobolev inequality in Lebesgue spaces presented in  \cite{[Maly]} to prove this type of inequality. Here, we will present a new proof to this type of inequality in Lemma \ref{psillemma} via the Young inequality in Lorentz  spaces rather than utilization of a median of function.  It seems that this proof is more elementary, self-contained and  short  than Mal\'y's proof.
 Moreover, general  Poincar\'e inequality  \eqref{gpil}   in Lorentz spaces are also obtained, which is independent of interesting. Boundedness of Riesz Transform in Lorentz spaces, \eqref{tiduu1} and  Biot-Savart law for the  incompressible flows imply \eqref{xduu1}. Parallelly, the generalized Biot-Savart law below
\be\label{gbsl}
\partial_{k}\partial_{k}u_{l}=\partial_{k}(\partial_{l}u_{k}+\partial_{k}u_{l}).
\ee
yields \eqref{dz1}.

For the  $\varepsilon$-regularity criteria at one scale, Barker \cite{[Barker]} obtained the following results by replacing the Lebesgue spaces  with Lorentz spaces in spatial direction and improved \eqref{ckn1} to
\be\label{Barker}
\|u\|_{L^{3}(-1,0;L^{3,\infty}(B(1)))}
+\|u\|_{L^{\f{3}{2}}(-1,0;L^{\f{3}{2},\infty}(B(1)))}\leq\varepsilon.
\ee
On the other hand,
authors in \cite{[HWZ]} recently generalized \eqref{ckn1} to
 \be\label{optical}
 \|u\|_{L^{p}(-1,0;L^{q}(B(1)))}+\|\Pi\|_{L^{1}(-1,0;L^{1}(B(1)))}\leq\varepsilon,~~1\leq 2/p+3/q <2, 1\leq p,\,q\leq\infty.
 \ee
As aforementioned, \eqref{optical} and  inclusion \eqref{Inclusion} yield the following results.
\begin{theorem}\label{the1.2}
		Let  the pair $(u,  \Pi)$ be a suitable weak solution to the 3D Navier-Stokes system \eqref{NS} in $Q(1)$.
		There exists an absolute positive constant $\varepsilon_{2}$
		such that if the pair $(u,\Pi)$ satisfies, for any 	$l_{1},~l_{2}>1 $ \be\label{0il}\|u\|_{L^{p,\infty},L^{q,\infty}(B(1))}+\|\Pi\|_{L^{l_{1},\infty}L^{l_{2},\infty}(B(1))}\leq\varepsilon_{2},~~1\leq 2/p+3/q <2, 1< p<\infty,\ee
		then, $u\in L^{\infty}(Q(1/2)).$
	\end{theorem}
\begin{remark}
Theorem \ref{the1.2} is a generalization of  \eqref{ckn1}, \eqref{Barker}, and \eqref{optical}.
\end{remark}
We give several applications of the  $\varepsilon$-regularity criteria in Lorentz spaces obtained above. Firstly, as a consequence of Theorem \ref{the1.1}, we have the following results.

\begin{coro}\label{coro}
 Suppose that $(u,\,\Pi)$ is a suitable weak solution to \eqref{NS}. Then there exists a positive constant $\varepsilon_{3}$ such that $(0,0)$ is a regular point provided that one  of the following four conditions holds
\begin{enumerate}[(1)]
 \item \label{coro1}
    $  u \in L^{p,\infty}(-1,0; L ^{q,\infty}(B(1)))$ and $$\|u\|_{L^{p,\infty}(-1,0; L ^{q,\infty}(B(1)))} \leq\varepsilon_{3}, ~~~~ \text{with} ~~~ 2/p+3/q=1, ~~2\leq p\leq\infty;  $$
 \item \label{coro2} $\nabla u \in L^{p,\infty}(-1,0; L ^{q,\infty}(B(1)))$ and  $$\|\nabla u\|_{L^{p,\infty}(-1,0; L ^{q,\infty}(B(1)))} \leq\varepsilon_{3}, ~~~~ \text{with} ~~~ 2/p+3/q=2, ~~ 1< p<\infty;  $$
 \item $\text{curl\,} u \in L^{p,\infty}(-1,0; L ^{q,\infty}(B(1)))$ and  $$\|\text{curl\,} u \|_{L^{p,\infty}(-1,0; L ^{q,\infty}(B(1)))} \leq\varepsilon_{3}, ~~~~ \text{with} ~~~ 2/p+3/q=2, ~~ 1< p<\infty;     $$
 \item $\mathcal{D}(u) \in L^{p,\infty}(-1,0; L ^{q,\infty}(B(1)))$ and  $$\|\mathcal{D}(u)\|_{L^{p,\infty}(-1,0; L ^{q,\infty}(B(1)))} \leq\varepsilon_{3}, ~~~~ \text{with} ~~~ 2/p+3/q=2, ~~ 1< p<\infty.
   $$\end{enumerate}
% then $(0,\,0)$ is regular point.
\end{coro}
\begin{remark}
To the knowledge of the authors, regularity criteria in terms of $\nabla u$ in Lorentz spaces to the Leary-Hopf weak solutions is due to He and Wang in \cite{[HW]}, where they showed that $\|\nabla u\|_{L^{p} (0,T;L^{q,\infty}( \mathbb{R}^{3}))}<\infty$ with $2/p+3/q=2$ ensures that $u$ is regular on $(0,T)$. Compared with results in \cite{[HW]},   result \eqref{coro2} in this corollary allows both the space and time directions be in   Lorentz spaces.
Authors in \cite{[JWW]}    showed the
whole space version of \eqref{coro2} to the Leray-Hopf weak solutions.
\end{remark}
\begin{remark}\label{remark1.4}
We point out that \eqref{coro1} in this corollary still holds for the Leary-Hopf weak solutions. Indeed,  a   Leary-Hopf weak  solution  $  u$ belonging to $   L^{p,\infty}(-1,0; L ^{q,\infty}(B(1)))$ with $2/p+3/q=1$ guarantees  that $u$ is a   suitable weak solution. This fact has been observed in \cite{[GKT],[WZ]}.
\end{remark}

Secondly, we turn our attention to the results involving
Leray's blow up rate in time.
In \cite{[Leray1]}, Leray proved that, for $q>3$ and sufficiently small $\varepsilon$, if smooth solution $u$ satisfies
\be\label{learay}
 \|u(\hs{\cdot},t)\|_{L^{q }(\mathbb{R}^{3})}\leq \varepsilon(-t)^{\f{3-q }{2q}}.
\ee
then $u$ is regular at $t=0$.

Very recently, Kukavica,   Rusin and    Ziane \cite{[KRZ]} improved \eqref{learay} to
 $$
 \|(u_{1},u_{2})\|_{L^{q }(\mathbb{R}^{3})}\leq M(\varepsilon)(-t)^{\f{3-q }{2q}} ~~~\text{and}~~~ \| u_{3}  \|_{L^{q }(\mathbb{R}^{3})}\leq \varepsilon(-t)^{\f{3-q }{2q}}.
$$
In the  spirit of   \cite{[KRZ]},  a generalization of Leray's blow up result is derived from Theorem \ref{the1.2}, that is
\begin{theorem}\label{the1.4}
For $q\geq3$, there exists an absolute positive constant  $\varepsilon$ such that if \hs{a}
smooth solution $u(x,t)$ satisfies
		 $$
 \|u (\cdot,t) \|_{L^{p,\infty}(\mathbb{R}^{3})}\leq \varepsilon(-t)^{\f{3-q }{2q}}.
$$
then the solution  $u(x,t)$ is regular at $t=0.$
\end{theorem}
Thirdly,
the structure of potential singular set $\mathcal{S}$ of solutions in \eqref{NS} attracted
extensive research such as  upper bound of box dimension or upper bound for the number of singular points $\mathcal{S}$( see e.g. \cite{[WZ],[CY19],[WW2],[Kukavica],[Neustupa],[CWY],[HWZ],[Seregin]}).
We set $\mathcal{S}(t)=\{(x,t)\in\mathcal{S}\}$. $N(t)$ represents the number of  $\mathcal{S}(t)$.
  In particular, Wang and Zhang \cite{[WZ]} showed that, for any $t\in(-1,0],$
\be\label{numofwz}
  N(t)\leq C \|u\|_{L^{p,\infty}(-1,0;L^{q}(\mathbb{R}^{3}))}~~\text{with}
  ~~2/p+3/q=1, 3<q<\infty.\ee
  Choe,  Wolf  and    Yang \cite{[CWY]} proved that
 $$
  N(t)\leq C \|u\|_{L^{ \infty}(-1,0;L^{3,\infty}(\mathbb{R}^{3}))}.$$
Here, the finial  result is to improve \eqref{numofwz}.
\begin{theorem}\label{the1.5}
		Let $u$ be a Leray-Hopf weak  solution  and satisfy $u\in L^{p,\infty}(-1,0;L^{q,l}(\mathbb{R}^{3}))$ where $q\leq l<\infty$. Then, there
holds
\be\label{1.31}
N(t)\leq C \|u\|_{L^{p,\infty}(-1,0;L^{q,l}(\mathbb{R}^{3}))},\ee
where
$2/p+3/q=1$ and $3<q<\infty$.
	\end{theorem}

This  paper is organized as follows. In   section 2,
we  collect some  materials of  Lorentz spaces. Various
Poincar\'e-Sobolev inequality   are dealt with in these space.    Then, we establish some dimensionless decay estimates.
In section 3, we  complete the  proof of Theorem 1.1.
 Section 4
 is devoted to
proving Theorem 1.3 and 1.4.

\section{Notations and  a key auxiliary lemma} \label{section2}
First, we introduce some notations used in this paper.
Throughout this paper, we denote
\begin{align*}
     &B(x,\mu):=\{y\in \mathbb{R}^{3}||x-y|\leq \mu\}, && B(\mu):= B(0,\mu),\\
     &Q(z,r)=:B(x,r)\times (t-r^2,t), && Q(\mu):= Q(0,\mu). \end{align*}
 For $p\in [1,\,\infty]$, the notation $L^{p}(0,\,T;X)$ stands for the set of measurable functions $f(x,t)$ on the interval $(0,\,T)$ with values in $X$ and $\|f(\cdot,t)\|_{X}$ belonging to $L^{p}(0,\,T)$.
  For simplicity,     we write $$\|\cdot\| _{L^{p}L^{q}(Q(r))}:=\|\cdot\| _{L^{p}(-r^{2},0;L^{q}(B(r)))}  ~~ \text{and}~~~
 \|\cdot\| _{L^{p}(Q(r))}:=\|\cdot\| _{L^{p}L^{p}(Q(r))}.$$
 The classical Sobolev space $W^{k,2}(\Omega)$ is equipped with the norm $\|f\|_{W^{k,2}(\Omega)}=\sum\limits_{\alpha =0}^{k}\|D^{\alpha}f\|_{L^{2}(\Omega)}$. Let $W_{0}^{k,2}(\Omega)$ be the completion of $C^{\infty}_{0}(\Omega)$ in the norm of $W^{k,2}(\Omega)$, where the space  $C^{\infty}_{0}(\Omega)$ is the smooth compactly supported functions on $\Omega$.  $|E|$ represents the $n$-dimensional Lebesgue measure of a set $E\subset \mathbb{R}^{n}$.  We will use the summation convention on repeated indices.
 $C$ is an absolute constant which may be different from line to line unless otherwise stated in this paper.

\begin{definition}\label{defi}
		A  pair   $(u, \,\Pi)$  is called a suitable weak solution to the Navier-Stokes equations \eqref{NS} provided the following conditions are satisfied,
		\begin{enumerate}[(1)]
			\item $u \in L^{\infty}(-T,\,0;\,L^{2}(\mathbb{R}^{3}))\cap L^{2}(-T,\,0;\,\dot{H}^{1}(\mathbb{R}^{3})),\,\Pi\in
			L^{3/2}(-T,\,0;L^{3/2}(\mathbb{R}^{3}));$\label{SWS1}
			\item$(u, ~\Pi)$~solves (\ref{NS}) in $\mathbb{R}^{3}\times (-T,\,0) $ in the sense of distributions;\label{SWS2}
			\item$(u, ~\Pi)$ satisfies the following inequality, for a.e. $t\in[-T,0]$,
			\begin{align}
				&\int_{\mathbb{R}^{3}} |u(x,t)|^{2} \phi(x,t) dx
				+2\int^{t}_{-T}\int_{\mathbb{R} ^{3 }}
				|\nabla u|^{2}\phi  dxds\nonumber\\ \leq&  \int^{t}_{-T }\int_{\mathbb{R}^{3}} |u|^{2}
				(\partial_{s}\phi+\Delta \phi)dxds
				+ \int^{t}_{-T }
				\int_{\mathbb{R}^{3}}u\cdot\nabla\phi (|u|^{2} +2\Pi)dxds, \label{loc}
			\end{align}
			where non-negative function $\phi(x,s)\in C_{0}^{\infty}(\mathbb{R}^{3}\times (-T,0) )$.\label{SWS3}
		\end{enumerate}
	\end{definition}
Next, we present some basic facts on Lorentz spaces. Recall that the distribution function of a   measurable function  $f$ on $\Omega$ is defined by
$$
f_{\ast}(\alpha)=|\{x\in \Omega:|f(x)|>\alpha\}|.
$$
The decreasing rearrangement of $f$ is the function $f^{\ast}$ defined   by
$$
f^{\ast}(t)=\inf\{\alpha>0;f_{\ast}(\alpha)\leq t\}.
$$
For $p,q\in[1,\infty]$, we define
$$
\|f\|_{L^{p,q}(\Omega)}=\left\{\ba
&\B(\int_{0}^{\infty}(t^{\f{1}{p}}f^{\ast}(t))^{q}\f{dt}{t}\B)^{\f1q}, ~~~q<\infty, \\
 &\sup_{t>0}t^{\f{1}{p}}f^{\ast}(t) ,~~~q=\infty.
\ea\right.
$$
Furthermore,
$$
L^{p,q}(\Omega)=\big\{f: f~ \text{is measurable function on}~ \Omega ~\text{and} ~\|f\|_{L^{p,q}(\Omega)}<\infty\big\}.
$$
Notice that
 identity definition of  Lorentz norm  can be found in \cite{[Grafakos],[Maly]}. Indeed, for $0<p<\infty$ and $0<q\leq\infty$, there holds
$$
\|f\|_{L^{p,q}(\Omega)}=\left\{\ba
&\B(p\int_{0}^{\infty}\alpha^{q}f_{*}(\alpha)^{\f{q}{p}}\f{d\alpha}{\alpha}\B)^{\f{1}{q}} , ~~~q<\infty, \\
 &\sup_{\alpha>0}\alpha f_{*}(\alpha)^{\f{1}{p}} ,~~~q=\infty.
\ea\right.
$$			
Similarly, one can defined
Lorentz spaces $L^{p,q}(0,T;X)$ in time for $p\leq q\leq\infty$. $f\in L^{p,   q}(0,T;X)$ means that $\|f\|_{L^{p,q}(0,T;X)}<\infty$, where
$$\|f\|_{L^{p,q}(0,T;X)}=\left\{\ba
&\B(p\int_{0}^{\infty}\alpha^q|\{t\in[0,T)
:\|f(t)\|_{X}>\alpha\}|^{\f{q}{p}}\f{d\alpha}{\alpha}\B)^{\f{1}{q}} , ~~~q<\infty, \\
 &\sup_{\alpha>0}\alpha|\{t\in[0,T)
:\|f(t)\|_{X}>\alpha\}|^{\f{1}{p}} ,~~~q=\infty.\ea\right.
$$
We list the properties of Lorentz spaces.
\begin{itemize}

\item Interpolation characteristic of Lorentz spaces \cite{[BL]}
\be\label{Interpolation characteristic}
(L^{p_{0},q_{0}}(\Omega),L^{p_{1},q_{1}}(\Omega))_{\theta,q}=L^{p,q}(\Omega)
~~~~\text{with}~~~ \f{1}{p}=\f{1-\theta}{p_{0}}+\f{\theta}{p_{1}},~0<\theta<1.\ee

\item
Scaling in  Lorentz spaces
$$\|f(\lambda x)\|_{L^{p,q}(\mathbb{R}^{n})}=\lambda^{-\f{n}{p}}\|f(\lambda x)\|_{L^{p,q}(\mathbb{R}^{n})}.$$
\item
Boundedness of Riesz Transform in Lorentz spaces \cite{[CF]}
\be\|R_{j}f\|_{L^{p,q}(\mathbb{R}^{n})}\leq C\| f\|_{L^{p,q}(\mathbb{R}^{n})},~1<p<\infty.\label{brl}\ee

\item
H\"older's inequality in Lorentz spaces  \cite{[Neil]}
 $$\ba
 &\|fg\|_{L^{r,s}(\Omega)}\leq \|f\|_{L^{r_{1},s_{1}}(\Omega)}\|g\|_{L^{r_{2},s_{2}}(\Omega)},
\\
&\f{1}{r}=\f{1}{r_{1}}+\f{1}{r_{2}},~~\f{1}{s}=\f{1}{s_{1}}+\f{1}{s_{2}}.
\ea$$

\item

The Lorentz spaces increase as the exponent $q$ increases \cite{[Grafakos],[Maly]}

For $1\leq p\leq\infty$ and $1\leq q_{1}<q_{2}\leq\infty,$
$$
\|f\|_{L^{p,q_{2}}(\Omega)}\leq \B(\f{q_{1}}{p}\B)^{\f{1}{q_{1}}-\f{1}{q_{2}}}\|f\|_{L^{p,q_{1}}(\Omega)}.
$$
\item
Inclusion on bounded domains in Lorentz spaces \cite{[Grafakos],[Maly]}

For any $1\leq m<M\leq\infty$,
\be\label{Inclusion}
\|f\|_{L^{m,r}(\Omega)}\leq \B(\f{1}{m}\B)^{\f{r-1}{r}}
\B(\f{q}{M}\B)^{\f{1}{q}}\f{|\Omega|
^{\f{1}{m}-\f{1}{M}}}{\f{1}{m}-\f{1}{M}}
\|f\|_{L^{M,q}(\Omega)}.
\ee
\item Sobolev inequality in Lorentz spaces \cite{[Neil],[Tartar]}
\be\label{sl}
\|f\|_{L^{\f{np}{n-p},p}(\mathbb{R}^{3})}\leq \|\nabla f\|_{L^{p}(\mathbb{R}^{3})}~~\text{with}~~1\leq q<n.\ee
\item Young inequality in Lorentz spaces \cite{[Neil]}

Let $1<p,q,r<\infty$, $0<s_{1},s_{2}\leq\infty$
,$\f{1}{p}+\f{1}{q}=\f{1}{r}+1$, and $ \f{1}{s}=\f{1}{s_{2}}+\f{1}{s_{1}}$. Then there holds
\be\label{young}
\|f\ast g\|_{L^{r,s}(\mathbb{R}^{n})}\leq\|f \|_{L^{p,s_{1}}(\mathbb{R}^{n})}\|g \|_{L^{q,s_{2}}(\mathbb{R}^{n})}.
\ee
\end{itemize}

Next, we turn our attention to Poincar\'e-Sobolev inequality in Lorentz spaces.
Mal\'y first mentioned this inequality and its proof in \cite{[Maly]}. We will give a new proof here.
\begin{lemma}[Poincar\'e-Sobolev inequality in Lorentz spaces]\label{psillemma}Suppose that
$1<p<n$ and $1\leq q\leq\infty$. Then
\begin{align}
 &\|f-\overline{f}\|_{ L^{\f{pn}{p-n},q}  (B(\varrho) )} \leq C\|\nabla f\|_{L^{p,q}(B( \varrho))}.\label{psil1}
\end{align}
\end{lemma}
\begin{proof}
Let  $\eta_{B(\varrho)}(y)$ be the characteristic function of $B(\varrho)$.
For $x\in B(\varrho)$, after a few computations, we discover that
$$\ba
|f(x)-\overline{f}_{B(\varrho)}|&\leq\f{1}{|B(\varrho)|}
\int_{B(\varrho)}|f(x)-f(y)|dy\\
&\leq C\int_{B(\varrho)}\f{|\nabla f(y)|}{|y-x|^{n-1}}dy\\
&=C\int_{\mathbb{R}^{n}}\eta_{B(\varrho)}(y)|\nabla f(y)||y-x|^{1-n} \eta_{B(2\varrho)}(x-y)dy\\
&=C h\ast g(x),
\ea$$
where
$h(x)=\eta_{B(\varrho)}(x)|\nabla f(x)|,g(x)=|x|^{1-n} \eta_{B(2\varrho)}(x)$.

This leads to
\be\label{lem2.1.2.6}
\|f(x)-\overline{f}_{B(\varrho)}\|
_{L^{ {\f{pn}{ n-p},q}}((B(\varrho)))}\leq C \|h\ast g \|
_{L^{ L^{\f{pn}{ n-p},q}}(B(\varrho))}\leq C \|h\ast g \|
_{L^{ L^{\f{pn}{ n-p},q}}(\mathbb{R}^{n})}.
\ee
Utilizing the Young inequality \eqref{young} in Lorentz spaces, we further derive that
 $$\ba
\|f(x)-\overline{f}_{B(\varrho)}\|
_{L^{ {\f{pn}{ n-p},q}}(B(\varrho))}\leq& C \|g  \|
_{L^{\f{n}{n-1},\infty}(\mathbb{R}^{n})}\|h  \|
_{L^{p,q}(\mathbb{R}^{n})}\\
\leq& C \|\nabla f  \|
_{L^{ p,q }(B(\varrho))}.\ea
$$
This completes the proof of this lemma.
\end{proof}
The above proof also implies Poincar\'e  inequality \eqref{gpil} below in Lorentz spaces. The case $n\leq q<\infty$ is an immediate   consequence of Lemma \ref{psillemma}.
\begin{lemma}[Poincar\'e  inequality in Lorentz spaces]\label{pil}Suppose that
$1<p<\infty$ and $1\leq q\leq\infty$. Then, for $B(\varrho)\subset \mathbb{R}^{n}$ with $n>2$,
\begin{align}
 &\|f-\overline{f}\|_{ L^{p,q}  (B(\varrho) )} \leq C\varrho\|\nabla f\|_{L^{p,q}(B( \varrho))}.\label{gpil}
\end{align}
\end{lemma}
\begin{proof}
We still use the notions in Lemma \ref{psillemma}.
We denote $p^{\natural}$ satisfying $$1+\f{1}{p}=\f{n-1}{n}+\f{1}{p^{\natural}}.$$
With the help of \eqref{lem2.1.2.6},
 the Young inequality \eqref{young} in Lorentz norm and \eqref{Inclusion}, we infer that
 $$\ba
\|f(x)-\overline{f}_{B(\varrho)}\|
_{L^{ L^{p,q}}(B(\varrho))}
\leq&
C \|h\ast g \|
_{L^{ L^{p,q}}(\mathbb{R}^{n})}\\
\leq& C \|g  \|
_{L^{\f{n}{n-1},\infty}(\mathbb{R}^{n})}\|h  \|
_{L^{p^{\natural},q}(\mathbb{R}^{n})}\\
\leq&  C\varrho\|g  \|
_{L^{\f{n}{n-1},\infty}(\mathbb{R}^{n})}\|h  \|
_{L^{p,q}(\mathbb{R}^{n})}\\
\leq& C \varrho\|\nabla f  \|
_{L^{p,q}(B(\varrho))}.\ea
$$
We finish the proof of this lemma.
\end{proof}

\begin{remark}\label{rmk2.1}
By
localization of \eqref{sl} via usual cut-off function as \cite{[WWZ]} and  a slight variant of the proof of Lemma \ref{zcl}-\ref{presure} below, following the path of \cite{[WWZ]}, one can achieve the proof of (\ref{them1.1}) in Theorem \ref{the1.1} without the utilization of Poincar\'e-Sobolev inequality \ref{psil1}  in Lorentz spaces.
\end{remark}
\begin{lemma}\label{zcl}  Let $q \in[2,6]$ and $\f{2}{p}+\f{3}{q}=\f{3}{2}$.  Then, for $\varrho>0$, there exists a constant $C$ such that
\begin{align}
&\|f-\overline{f}_{\varrho}\|_{L^{p}L^{q,2}   (Q(\varrho) )} \leq C\|\nabla f\|_{L^{2}(Q( \varrho ))}^{\f32-\f{3}{q}}\|f\|
^{\f{3}{q}-\f{1}{2}}_{L^{\infty}L^{2}(Q( \varrho))}.
\label{locsl}
\end{align}
\end{lemma}
\begin{proof}
For $2<q<6$, by the interpolation characteristic \eqref{Interpolation characteristic} or the H\"older inequality in Lorentz spaces,  we know that there exists a constant $0<\theta<1$ such that
$$
\|f-\overline{f}_{\varrho}\|_{L^{q,2}(B(\varrho))} \leq C \|f-\overline{f}_{\varrho}\|^{\theta}_{L^{2}(B(\varrho))}
\|f-\overline{f}_{\varrho}\|^{1-\theta}_{L^{6,2}(B(\varrho)) }~~\text{with} ~~ \f{1}{q}=\f{\theta}{2}+\f{1-\theta}{6}.
$$
It is clear that the above inequality valid for $q=2$ and $q=6$.
Summarily, we always have, for $2\leq q\leq6$,
$$
\|f-\overline{f}_{\varrho}\|_{L^{q,2}(B(\varrho))} \leq C \|f\|^{\f{6-q}{2q}}_{L^{2}(B(\varrho))}
\|f-\overline{f}_{\varrho}\|^{\f{3q-6}{2q}}_{L^{6,2}(B(\varrho)) }.
$$
Furthermore, we derive from  that
\be\label{GNL}
\|f-\overline{f}_{\varrho}\|_{L^{q,2}(B(\varrho))} \leq C \|f\|^{\f{6-q}{2q}}_{L^{2,2}(B(\varrho))}
\|\nabla f\|^{\f{3q-6}{2q}}_{L^{2}(B(\varrho)) }.
\ee
Combining this and $\f{2}{p}+\f{3}{q}=\f{3}{2}$ yields \eqref{locsl}.
This achieves the proof of the desired estimate.
\end{proof}

 Before proceeding  further, according to the natural scaling in \eqref{NS},  we  write  the following dimensionless quantities,
\begin{align}
&E_{\ast}(\varrho)=\frac{1}{\varrho}\iint_{ Q(\varrho) }|\nabla u|^2dx dt,& E(\varrho)=\sup_{-\varrho^2\leq   t<0}\frac{1}{\varrho}\int_{B(\varrho)}|u|^2dx,\nonumber\\
&E_{3}(\varrho)=\frac{1}{\varrho^{2}}\iint_{ Q(\varrho) }|u|^{3}dx dt,&D_{3/2}(\varrho)=\frac{1}{\varrho^{2}}\iint_{ Q(\varrho)} |\Pi-\bar{\Pi}_{ B(\varrho) }|^{\frac{3}{2}}dx
dt. \nonumber
\end{align}
In the proof of Theorem \ref{the1.1}, we also require
\begin{align}
   E_{p;q,\infty}(\varrho)=
\varrho^{-1}\|u-\overline{u}_{\varrho}\|_{L^{p}_t
L^{q,\infty}_{x}(Q(\varrho))}~~\text{with}~~
 \f{2}{p}+ \frac{3}{q}=2.
\nonumber\end{align}
\begin{lemma}\label{ineq}
For $0<\mu\leq \rho$,~
there is an absolute constant $C$  independent of  $\mu$ and $\rho$,~ such that
 \begin{align}
\wred{E_{3}(\mu)\leq C\B(\f{\rho}{\mu}\B)^{2}E_{p;q,\infty} (\rho) E_{\ast}(\rho)^{1-\f{1}{p}} E^{\f{1}{p}}(\rho) +C\B(\f{\mu}{\rho}\B)E_{3}(\rho).}
\label{ineq2/2}    \end{align}
\end{lemma}
\begin{proof}
By using the H\"older  inequality and \eqref{locsl}, we get
\be\ba\label{key2.9} \iint_{Q(\varrho)} |u- \overline{u}_{ \varrho}|^{3}dxds&=\iint_{Q(\varrho)} |u- \overline{u}_{ \varrho}| |u- \overline{u}_{ \varrho}|^{2}dxds
\\& \leq C \|u- \overline{u}_{ \varrho}\|_{ L_{t}^{p}L_x^{q,\infty} (Q(\varrho))} \|u- \overline{u}_{ \varrho}\|^{2}_{ L_{t}^{2p^{\ast}}L_x^{2q^{\ast},2} (Q(\varrho))}
\\& \leq C \|u- \overline{u}_{ \varrho}\|_{ L_{t}^{p}L_x^{q,\infty} (Q(\varrho))} \|\nabla u\|_{L^{2}(Q(\sqrt{2}\varrho ))}^{\f{2}{p^{\ast}}}\|u \|^{2-\f{2}{p^{\ast}}}_{L^{\infty}_{t}L_x^{2}(Q(\varrho))}. \ea\ee
 Thanks to  $\f{2}{p}+\f{3}{q}=2$, we see that
\be\ba
 \iint_{Q( \varrho)}|u- \overline{u}_{ \varrho}|^{3}dxdt
\leq  C\|u- \overline{ u}_{\varrho }\|_{L_{t}^{p}L_x^{q,\infty}(Q(\varrho))}   \|\nabla u\|_{L^{2}( Q(\varrho) )}^{ \f{1}{q} } \|u- \overline{ u}_{\varrho} \|^{2-\f{1}{q}}_{ L^{\infty}L^{2} (Q(\varrho))}.\label{eq3.4}
\ea\ee
  In view of   the triangle inequality,  we observe that
\begin{align}\nonumber
\iint_{Q(\mu)}|u|^{3}dx\leq& C\iint_{Q(\mu)}|u-\bar{u}_{{\rho}}|^{3}dx
+C\iint_{\wred{Q(\mu)}}|\bar{u}_{\rho}|^{3} dx\\
\leq& C\iint_{Q(\rho)}|u-\bar{u}_{\rho}|^{3}dx
 +
 C\f{\mu^{3}}{\rho^{3}}\B( \iint_{\wred{Q(\rho)}}|u|^{3}dx\B). \label{lem2.31}
 \end{align}
Plugging \eqref{eq3.4} into \eqref{lem2.31}, we know that  \eqref{ineq2/2}.
 \end{proof}
 \begin{lemma}\label{presure}
For $0<8\mu\leq \rho$, there exists an absolute constant $C$  independent of $\mu$ and $\rho$ such that
\begin{align}
&D_{3/2}(\mu)\leq
C\left(\f{\rho}{\mu}\right)
^{2} E_{p;q,\infty } (\rho)\wred{E_{\ast}(\rho)^{1-\f{1}{p}} E^{\f{1}{p}}(\rho)}
+C\left(\f{\mu}{\rho}\right)^{\f{5}{2}}D_{3/2}(\rho).\label{pe}
\end{align}
where the pair $(p,q)$ is defined in \eqref{u11}.
\end{lemma}
\begin{proof}
To localize the pressure equation,
we invoke the usual cut-off function $\phi\in C^{\infty}_{0}(B(\f{\rho}{2}))$ such that $\phi\equiv1$ on $B(\f{ \rho}{2})$ with $0\leq\phi\leq1$,
$|\nabla\phi |\leq C\rho^{-1} $ and $ ~|\nabla^{2}\phi |\leq
C\rho^{-2}.$

By applying  the incompressible condition, the pressure equation can be  reformulated as
$$
\partial_{i}\partial_{i}(\Pi\phi)=-\phi \partial_{i}\partial_{j} U_{i,j}
+2\partial_{i}\phi\partial_{i}\Pi+\Pi\partial_{i}\partial_{i}\phi
,$$
where $U_{i,j}=(u_{j}- \bar{u}_{{\rho}})(u_{i}-\bar{u}_{{\rho}})$. Before going further, we denote  $\Phi$  stands for the standard normalized fundamental solution of Laplace equation in $\mathbb{R}^{3}$. Then, there holds, for $x\in B(\f{3\rho}{3})$,
\be\ba\label{pp}
\Pi(x)=&\Phi\ast \{-\phi \partial_{i}\partial_{j} U_{i,j}
+2\partial_{i}\phi\partial_{i} \Pi + \Pi \partial_{i}\partial_{i}\phi
\}\\
=&-\partial_{i}\partial_{j}\Phi \ast (\phi  U_{i,j} )\\
&+2\partial_{i}\Phi \ast(\partial_{j}\phi U_{i,j} )-\Phi \ast
(\partial_{i}\partial_{j}\phi U_{i,j} )\\
&  +2 \partial_{i}\Phi \ast(\partial_{i}\phi \Pi) -\Phi \ast(\partial_{i}\partial_{i}\phi \Pi)\\
 =:  &P_{1}(x)+P_{2}(x)+P_{3}(x).
\ea\ee
Noting the fact that $\phi(x)=1, $ where $x\in B(\mu)$
  ($0<\mu\leq\f{\rho}{4} $), we know that
\[
\Delta(P_{2}(x)+P_{3}(x))=0.
\]
 According to the interior  estimate of harmonic function
and the H\"older  inequality, we thus have, for every
$ x_{0}\in B(\f{\rho}{8} )$,
\be\ba
| (P_{2}+P_{3})(x_{0})|&\leq \f{C}{\rho^{4}}\|(P_{2}+P_{3})\|_{L^{1}(B_{x_{0}}(\f{\rho}{8}))}
\\
&\leq \f{C}{\rho^{4}}\|(P_{2}+P_{3})\|_{L^{1}(B(\f{\rho}{4}))}\\
&\leq \f{C}{\rho^{4}}\rho^{3(1-\f{1}{q})} \|(P_{2}+P_{3})\|_{ L^{3/2} (B(\f{\rho}{4}))}.\label{lem2.4.2.15}
\ea\ee
We infer from \eqref{lem2.4.2.15} that
\be \|  P_{2}+P_3 \|^{3/2}_{L^{\infty}(B(\f{\rho}{8}))}\leq C \rho^{-9/2}\|P_2+P_3\|^{3/2}_{L^{3/2}(B(\f{\rho}{4}))} .\label{lem2.4.2.16}\ee
Using the  mean value theorem  and \eqref{lem2.4.2.16} , for any  $\mu\leq \f{\rho}{4}$, we arrive at
\be\ba\label{lem2.4.2.17}
\|(P_{2}+P_{3})-\overline{(P_{2}+P_{3})}_{\mu}\|^{3/2}_{L^{3/2}(B(\mu))}\leq&
C\mu^{3} \|(P_{2}+P_{3})-\overline{(P_{2}+P_{3})}_{\mu}\|^{ 3/2 }_{L^{\infty}(B(\mu))}\\
\leq& C
\mu^{3} (2\mu)^{ 3/2 }\|\nabla (P_{2}+P_{3})\|^{ 3/2 }_{L^{\infty}(B(\f{\rho}{4 }))}\\
\leq& C\Big(\f{\mu}{\rho}\Big)^{\f{9}{2}}\|(P_{2}+P_{3})\|^{3/2}_{L^{3/2}
(B(\f{\rho}{4}))}.
\ea\ee
By time integration, we get
$$
\|(P_{2}+P_{3})-\overline{(P_{2}+P_{3})}_{\mu}\|
^{\f{3}{2}}_{L^{\f{3}{2}}(Q(\mu))}\leq C\Big(\f{\mu}{\rho}\Big)^{ \f{9}{2}}
\|(P_{2}+P_{3})\|^{\f{3}{2}}_{L^{\f{3}{2}}(Q(\f{\rho}{4}))}.
$$
As $(P_{2}+P_{3})-((P_{2}+P_{3}))_{B(\f{\rho}{4})}$ is also a Harmonic function  on $B(\f{\rho}{4})$, we deduce taht
$$\ba
&\|(P_{2}+P_{3})-\overline{(P_{2}+P_{3})}_{\mu}\|
^{\f{3}{2}}_{L^{3/2}(Q(\mu))}
\\
\leq & C\Big(\f{\mu}{\rho}\Big)^{ \f{9}{2}}
\|(P_{2}+P_{3})-\overline{(P_{2}+P_{3})}_{\f{\rho}{4}}\|
^{\f{3}{2}}
_{L^{\f{3}{2}}(Q(\f{\rho}{4}))}.
\ea$$
The triangle inequality  further allows us to get
$$\ba
&\|(P_{2}+P_{3})-\overline{(P_{2}+P_{3})}_{(\f{\rho}{4})}\|_{L^{\f{3}{2}}(Q(\f{\rho}{2\sqrt{6}}))}\\
\leq& \|\Pi-\overline{\Pi}_{(\f{\rho}{4})}\|_{L^{\f{3}{2}}(Q(\f{\rho}{4}))}
+\|P_{1}-\overline{P_{1}}_{(\f{\rho}{4})}\|_{L^{\f{3}{2}}(Q((\f{\rho}{4})))}
\\
\leq& C \| \Pi-\overline{\Pi}_{\rho} \|_{L^{\f{3}{2}}(Q(\f{\rho}{4}))}
+C\|P_{1}\|_{L^{\f{3}{2}}(Q(\f{\rho}{4}))},
\ea$$
which in turn yields that
\be\label{p2rou}\ba
&\|(P_{2}+P_{3})-\overline{(P_{2}+P_{3})}_{\mu}\|
^{\f{3}{2}}_{L^{\f{3}{2}}(Q(\mu))}\\
\leq& C\Big(\f{\mu}{\rho}\Big)^{ \f{9}{2}}\Big(\|\Pi-\overline{\Pi}
_{(\rho)}\|^{\f{3}{2}}_{L^{\f{3}{2}}(Q(\f{\rho}{4}))}
+\|P_{1}\|^{\f{3}{2}}_{L^{\f{3}{2}}(Q(\f{\rho}{4}))}\Big)
\\
\leq& C\Big(\f{\mu}{\rho}\Big)^{ \f{9}{2}}\Big(\|\Pi-\overline{\Pi}
_{(\rho)}\|^{\f{3}{2}}_{L^{\f{3}{2}}(Q(\rho))}
+\|P_{1}\|^{\f{3}{2}}_{L^{\f{3}{2}}(Q(\f{\rho}{4}))}\Big).
\ea
\ee
By virtue of the H\"older inequality and the argument in \eqref{key2.9}, we get
$$\ba
&\iint_{Q(\rho/2)}|u-\bar{u}_{ \rho}|^{3}dxds\\
 \leq& C \|u-  \bar{u}_{ \rho}\|_{L^{p}L^{q,\infty}(Q(\rho))}  \|\nabla u\|_{L^{2} (Q(\rho))}^{2-\f{2}{p} } \|u\|^{\f{2}{p}}_{L^{\infty}L^{2}(Q(\rho))}.
 \ea$$
The classical Calder\'on-Zygmund Theorem and the latter inequality implies that
\be\label{lem2.4.2}\ba
\iint_{Q(\f{\rho}{4})}|P_{1}(x)|^{\f{3}{2}}dxds
\leq& C \iint_{Q(\f{\rho}{2})}|u-\overline{u}_{\rho }|^{3}
dx\\
\leq&  C  \|u-\overline{u}_{\rho }\|_{L^{p}L^{q}(Q(\rho))}  \|\nabla u\|_{L^{2} (Q(\rho))}^{2-\f{2}{p} } \|u\|^{\f{2}{p}}_{L^{\infty}L^{2}(Q(\rho))},
 \ea\ee
and
\be\label{lem2.4.3}\ba
\iint_{Q(\mu)}|P_{1}(x)|^{\f{3}{2}}dx& \leq C  \|u- \overline{u}_{\rho }\|_{L^{p}L^{q,\infty}(Q(\rho))}  \|\nabla u\|_{L^{2} (Q(\rho))}^{2-\f{2}{p} } \|u\|^{\f{2}{p}}_{L^{\infty}L^{2}(Q(\rho))}.
 \ea\ee
The  inequalities \eqref{p2rou}-\eqref{lem2.4.3} allow  us to deduce that
\be\label{lem2.3}
\ba
\iint_{Q(\mu)}|\Pi-\Pi_{\mu}|^{\f{3}{2}}dxds \leq& C\iint_{Q(\mu)}
|P_{1}-(P_{1})_{\mu}|^{\f{3}{2}}
+ \big|P_{2}+P_{3}-(P_{2}+P_{3})_{\mu}\big|^{\f{3}{2}} dx \\
\leq& C \|u- \overline{u}_{\rho }\|_{L^{p}L^{q,\infty}(Q(\rho))}  \|\nabla u\|_{L^{2} (Q(\rho))}^{2-\f{2}{p} } \|u\|^{\f{2}{p}}_{L^{\infty}L^{2}(Q(\rho))}
 \\
& +C\left(\f{\mu}{\rho}\right)
^{\f{9}{2}}\int_{B(\rho)}|\Pi-\Pi_{\rho}|^{\f{3}{2}}.
\ea
\ee
We readily get
\be\label{lem2.42}
\ba
\f{1}{\mu^2}\iint_{Q(\mu)}|\Pi-\Pi_{\mu}|^{\f{3}{2}}
\leq&   C\f{1}{\mu^2} \|u- \overline{u}_{\rho }\|_{L^{p}L^{q,\infty}(Q(\rho))}  \|\nabla u\|_{L^{2} (Q(\rho))}^{2-\f{2}{p} } \|u\|^{\f{2}{p}}_{L^{\infty}L^{2}(Q(\rho))}\\
& +C\left(\f{\mu}{\rho}\right)
^{\f{5}{2}}
\f{1}{\rho^{2}}\iint_{Q(\rho)}|\Pi-\overline{\Pi}_{\rho}|^{\f{3}{2}}dx,
\ea
\ee
which leads to
\begin{align}
D_{3/2}(\mu)\leq &
C\left(\f{\rho}{\mu}\right)
^{2} E_{p;q,\infty}(\rho)E_{\ast}(\rho)^{1-\f{1}{p}} E^{\f{1}{p}}(\rho)
+C\left(\f{\mu}{\rho}\right)^{\f{5}{2}}D_{3/2}(\rho).\label{presure4}
\end{align}
The proof of this lemma is   completed.
 \end{proof}

To prove Theorem  \ref{the1.5}, we write
$$\ba
&E_{p;q,l}(\varrho)=\varrho^{-[(3-q)\f{p}{q}+2]}\|u\|^{p}_{L^{p}
(-\varrho^{2},0;L^{q,l}(B(\varrho)))},\\
&D_{p;q,l}(\varrho)=\varrho^{-[(3-2q)\f{p}{q}+2]}\|\Pi\|^{p/2}_{L^{p}
( -\varrho^{2},0;L^{q,l}(B(\varrho)))}.\ea
$$
The following result in  Lebesgue spaces is due to \cite{[WZ]}. Here, we generalize it to allow the space direction belonging to Lorentz spaces.
\begin{lemma}\label{presure2}
For $0<8\mu\leq \rho$, there exists an absolute constant $C$  independent of $\mu$ and $\rho$ such that
\begin{align}
  \label{2.29}
D_{\f{p}{2};\f{q}{2},l}(r)\leq C\big(\frac{\rho}{r}\big)^{\frac{3 p}{q}+2-p}E_{p;q,l}(\varrho)
+C\big(\frac{r}{\rho}\big)^{(2-\frac{4}{p})\f{p}{2}}
D_{\f{p}{2};\f{q}{2},l}(\rho).
\end{align}
\end{lemma}
\begin{proof}
We still use the representation of pressure $\Pi$ in \eqref{pp}.
  The interior  estimate of harmonic function
and H\"older's inequality ensures that, for every
$ x_{0}\in B(\rho/8)$,
$$\ba
|  (P_{2}+P_{3})(x_{0})|&\leq \f{C}{\rho^{3}}\|(P_{2}+P_{3})\|_{L^{1}(B_{x_{0}}(\rho/8))}
\\
&\leq \f{C}{\rho^{3}}\|(P_{2}+P_{3})\|_{L^{1}(B(\rho/4))}\\
&\leq \f{C}{\rho^{3}}\rho^{3(1-\f{1}{q})} \|(P_{2}+P_{3})\|_{L^{q,l}(B(\rho/4))}
,
\ea$$
which in turn implies
$$\|  P_{2}+P_{3}\|^{q}_{L^{\infty}(B(\rho/8))}\leq C \rho^{-3}\|  P_{2}+P_{3}\|^{q}_{L^{q,l}(B(\rho/4))}.$$
The latter inequality  leads to, for any  $\mu\leq \f{1}{8}\rho$,
\be\ba\label{needinlb}
\|(P_{2}+P_{3})\|_{L^{q,l}(B(\mu))}\leq&
C\mu^{\f{3}{q}} \|(P_{2}+P_{3})\|_{L^{\infty}(B(\mu))}\\
\leq& C
\mu^{\f{3}{q}}   \| (P_{2}+P_{3})\|_{L^{\infty}(B(\rho/8))}\\
\leq& C\Big(\f{\mu}{\rho}\Big)^{\f{3}{q}}\|(P_{2}+P_{3})\|_{L^{q,l}
(B(\rho/4))}.
\ea\ee
Integrating   in time on
$(-\mu^{2 },\,0)$  , we infer that
$$
\|P_{2}+P_{3}
\|_{L^{p}L^{q,l}(Q(\mu))}\leq C\Big(\f{\mu}{\rho}\Big)^{\f{3}{q} }
\|(P_{2}+P_{3})\|_{L^{p}L^{q,l} (Q(\rho/4))}.
$$
In the light  of   the triangle inequality, we have
$$\ba
&\|(P_{2}+P_{3}) \|_{L^{p}L^{q,l}(Q(\rho/4))}
\leq  C(\|\Pi\|_{L^{p}L^{q,l}(Q(\rho/4))}
+\|P_{1} \|_{L^{p}L^{q,l}(Q(\rho/4))}).
\ea$$
As a consequence, we obtain
\be\label{p2rou2}\ba
 \|(P_{2}+P_{3}) \|
 _{L^{p}L^{q,l}(Q(\mu))}
\leq  C\Big(\f{\mu}{\rho}\Big)^{^{\f{3}{q}}}\Big(\|\Pi\|_{L^{p}L^{q,l}(Q(\rho))}
+\|P_{1}\|_{L^{p}L^{q,l}(Q(\rho/4))}\Big).
\ea
\ee
Due to the boundedness of Riesz Transform \eqref{brl} in Lorentz spaces, one has
\be\label{lem2.4.22}\ba
\|P_{1}\|_{L^{q,l}(B(\rho/4))}
\leq& C \||u|^{2}\|_{L^{q,l}(B(\rho/2))},
 \ea\ee
and
\be\label{lem2.4.32}\ba
\|P_{1}\|_{L^{q,l}(B(\mu))}
\leq& C \||u|^{2}\|_{L^{q,l}(B(\rho/2))}.
 \ea\ee
Using \eqref{p2rou2}-\eqref{lem2.4.32}, we zchieve
\be\label{lem2.32}
\ba
\|\Pi \|^{p}_{L^{p}L^{q,l}(Q(\mu))} \leq& C\|P_{1} \|^{p}_{L^{p}L^{q,l}(Q(\mu))} +\|P_{2}+P_{3} \|^{p}_{L^{p}L^{q,l}(Q(\mu))}
 \\
\leq& C \||u|^{2}\|^{p}_{L^{p}L^{q,l}(Q(\rho/2))}
  +C\left(\f{\mu}{\rho}\right)^{\f{3 p}{q}} \|\Pi \|^{p}_{L^{p}L^{q,l}(Q(\rho))}.
\ea
\ee
We thus get
\be\label{lem2.422}
\ba
\mu^{-p(-2+\f3q+\f2p)}\|\Pi \|^{p}_{L^{p}L^{q,l}(Q(\mu))} \leq& C \B(\f{\rho}{\mu}\B)^{p(-2+\f3q+\f2p)}\rho^{-p(-2+\f3q+\f2p)}\|u\|^{2p}_{L^{2p}L^{2q,2l}(Q(\rho/2))}
  \\&+C\left(\f{\mu}{\rho}\right)^{p(2-\f{2}{p})} \rho^{-p(-2+\f3q+\f2p)}\|\Pi \|^{p}_{L^{p}L^{q,l}(Q(\rho))}.
  \ea
\ee
which  means \eqref{2.29}.
The proof of this lemma is   completed.
\end{proof}
\section{$\varepsilon$-regularity criteria  in Lorentz spaces}
In this section, with the decay estimates in Lemma  \ref{ineq} and Lemma \ref{presure}, we begin with the proof of Theorem \ref{them1.1}.
\begin{proof}[Proof  of \eqref{them1.1} in Theorem \ref{the1.1}]
We derive
 from \eqref{u1}  that
 there is a constant $\varrho_0$ such that, for any $\varrho\leq \varrho_{0}$,
$$
 \varrho^{-1 }
\|u-\overline{ u_{\varrho} }\|_{L_{t}^{p}L_{x}^{q,\infty} (Q(\varrho))} \leq\varepsilon_{1}.
$$
Taking advantage of  the Young inequality and local energy inequality \eqref{loc}, we find
\be\label{eq:88}\ba
E(\rho)+E_{\ast}(\rho)\leq& C\Big[E^{2/3}_{3}(2\rho)+E_{3}(2\rho)
+ D_{3/2} (2\rho)\Big]\\
\leq& C\Big[1+E_{3}(2\rho)
+ D_{3/2} (2\rho)\Big].
\ea\ee
 \eqref{ineq2/2} in Lemma \ref{ineq} and  \eqref{eq:88} guarantees that, for
$2 \mu\leq\rho$,
\be\ba
E_{3}(\mu)\leq&C \left(\dfrac{\rho}{\mu}\right)^{2}
 E_{p;q,\infty} (\rho/2) E_{\ast}(\rho/2)^{1-\f{1}{p}} E^{\f{1}{p}}(\rho/2)
    +C\left(\dfrac{\mu}{\rho}\right)E_{3}(\rho/2)\\
    \leq&C \left(\dfrac{\rho}{\mu}\right)^{2}
 E_{p;q,\infty} (\rho/2)\B( 1+E_{3}(\rho)
+ D_{3/2} (\rho)\B)
    +C\left(\dfrac{\mu}{\rho}\right)E_{3}(\rho/2)
\\
    \leq&C \left(\dfrac{\rho}{\mu}\right)^{2}
 E_{p;q,\infty} (\rho )\B( 1+E_{3}(\rho)
+ D_{3/2} (\rho)\B)
    +C\left(\dfrac{\mu}{\rho}\right)E_{3}(\rho ).\label{3.2}    \ea\ee
  Moreover,   \eqref{pe} in Lemma \ref{presure} states that, for $8 \mu\leq\rho$,
\be\label{3.3}
D_{3/2}(\mu)\leq
C\left(\f{\rho}{\mu}\right)
^{2} E_{p;q,\infty} (\rho)\B(1+E_{3}(\rho)
+D_{3/2}(\rho)\B)
+C\left(\f{\mu}{\rho}\right)^{\f{5}{2}}D_{3/2}(\rho).
\ee
For notational convenience, we define
$$F(\mu)= E_{3}(\mu)+ D_{3/2}(\mu).$$
Then, thanks to \eqref{3.2}  and \eqref{3.3}, we find that
$$\ba
F(\mu)\leq&
C
\left(\dfrac{\rho}{\mu}\right)^{2}
 E_{p;q,\infty} (\rho)F(\rho)
    +C \left(\dfrac{\rho}{\mu}\right)^{2}
 E_{p;q,\infty} (\rho)
+C
\left(\dfrac{\mu}{\rho}\right)
 F(\rho)\\
\leq & C_{1}\lambda^{-2}\varepsilon_{1} F(\rho)+
C_{2}\lambda^{-2} \varepsilon_{1}  +C_{3}\lambda F(\rho),
\ea$$
 where  $\lambda=\f{\mu}{\rho}\leq \f{1}{8\sqrt{6}}$ and $\rho\leq   \varrho_{0}  $.\\
Choosing $\lambda,~\varepsilon_{1}$ such that $ \theta =2C_{3}\lambda<1$ and $\varepsilon_{1}=\min\{ \f{ \theta \lambda^{2}}{2C_{1}} ,
 \f{(1- \theta )\lambda^{2}\varepsilon}{2C_{2}\lambda^{-2}}
 \}$ where $\varepsilon$ is the constant in \eqref{optical}, we see that
\be\label{iter}
F(\lambda\rho)\leq  \theta F(\rho)+C_{2}\lambda^{-2} \varepsilon_{1}. \ee
By iterating  $\eqref{iter}$, we readily   get
 \[
F(\lambda^{k}\rho)\leq  \theta ^{k}F(\rho)+\f{1}{2}\lambda^{2}\varepsilon. \]
According to the definition of $F(r)$, for a fixed $\varrho_{0}>0$, we know that there exists a positive number $ K_{0}$~such that
$$ \theta ^{K_{0}}F(\varrho_{0}) \leq \f{M(\|u\|_{L^{\infty}L^{2}},\|u\|_{L^{2}W^{1,2}},
\|\Pi\|_{L^{3/2}L^{3/2}})}{\varrho_{0}^{2}} \theta ^{K_{0}}
\leq\dfrac{1}{2}\varepsilon \lambda^{2}.$$
We denote $\varrho_{1}:=\lambda^{K_{0}}\varrho_{0}$. Then,  for all $0<\varrho\leq \varrho_{1}$ , $\exists k\geq
K_{0}$,~such that $\lambda^{k+1}\varrho_{0}\leq \varrho\leq \lambda^{k} \varrho_{0}$, there holds
\[
 \begin{aligned}
& E_{3}(\varrho)+D_{3/2}(\varrho)\\
=&\frac{1}{\varrho^{2}}\iint_{Q(\varrho)}|  u|^3dxdt+
\frac{1}{\varrho^{2}}\iint_{Q(\varrho)}|\Pi-\overline{\Pi}_{\varrho}|^{\frac{3}{2}}dxdt\\
\leq&  \frac{1}{(\lambda^{k+1}\varrho_{0})^{2}}\iint_{Q(\lambda^{k}\varrho_{0})}|  u|^3dxdt
 +
\frac{1}{(\lambda^{k+1}\varrho_{0})^{2}}\iint_{Q(\lambda^{k}\varrho_{0})}|\Pi-\overline{\Pi}_{\lambda^{k}\varrho_{0}}|^{\frac{3}{2}}dxdt
\\
\leq & \f{1}{\lambda^{2}}F(\lambda^{k}\varrho_{0}) \\
\leq &\f{1}{\lambda^{2}}( \theta ^{k-K_{0}} \theta ^{K_{0}}
F(\varrho_{0})+\f{1}{2}\lambda^{2}\varepsilon )\\
\leq &\varepsilon.
 \end{aligned}
\]
This together with  \eqref{optical} completes the  proof of  Theorem \ref{the1.1}.
\end{proof}

\begin{proof}[Proof of \eqref{them1.2} in Theorem \ref{the1.1}]

According to   Poincar\'e-Sobolev inequality in Lorentz spaces \eqref{rmk2.1}, we know that
$$\ba
\|u-\overline{u}_{B(\varrho)}\|_{ L_{t}^{p^{\ast}}L_x^{q,\infty} (Q(\varrho))}  &\leq   C \|\nabla u \|_{ L_{t}^{p}L_x^{q,\infty} (Q(\varrho))}, \ea$$ where $$1+\f{3}{p^\ast}=\f{3}{p},~~1<p^{\ast}<3.$$
Thus,
$$\ba
\varrho^{-1}\|u-\overline{u}_{B(\varrho)}\|_{ L_{t}^{p^{\ast}}L_x^{q,\infty} (Q(\varrho))}  &\leq C \varrho^{-1}  \|\nabla u \|_{ L_{t}^{p}L_x^{q,\infty} (Q(\varrho))}.\ea$$
Combining this  and \eqref{u1}, one can complete the proof of \eqref{them1.2} in Theorem \ref{the1.1}.
\end{proof}
The key ingredient of the proof   of \eqref{3inth1.1}  in Theorem \ref{the1.1} is the following lemma.
\begin{lemma}\label{lemma3.1}
Suppose $ w=\text{curl}u \in L^{p}L^{q,\infty}(Q(1))$ with
$ \frac 2p+\frac 3q=3$ and $1 \le q < \infty$. Then, for any $0<\mu\leq\rho/4$, it holds
\begin{align}
  &G(p,q;\mu)\leq C \Big(\frac{\rho}{\mu}\Big)W(p,q;\mu)
  +C \Big(\frac{\mu}{\rho}\Big)^{\frac{3}{s}-1}  G(p,q;\rho),~~s>3.\label{lem3.1}
\end{align}
\end{lemma}

\begin{proof}
We first recall the Biot-Savart law
$$\Delta u=-\text{curl}\,w.$$
In particular, we note that
$$\Delta u_{1}=-(\partial_{2}w_{3}+\partial_{3}w_{2}). $$
 Direct calculation yields that
$$\Delta \partial_{k}u_{1}=-(\partial_{2}\partial_{k}w_{3}+\partial_{3}\partial_{k}w_{2}). $$
Arguing as we did in \eqref{pp}, we write
$$
\partial_{k}u_{1}=R_{2}R_{k}(\phi w_{3})+R_{3}R_{k}(-\phi w_{2})+H(x),$$
where $H(x)$ is a harmonic function.
A modified version of \eqref{needinlb}, we find that
$$\ba
\|H(x)\|_{L^{q,\infty}(B(\mu))}\leq& C \B(\f{\mu}{\rho}\B)^{3}\|H(x)\|_{L^{q,\infty}(B(\rho))}
\\\leq& C \B(\f{\mu}{\rho}\B)^{3}\|R_{2}R_{k}(\phi w_{3})+R_{3}R_{k}(-\phi w_{2})\|_{L^{q,\infty}(B(\rho))}
\\&+C \B(\f{\mu}{\rho}\B)^{3}\|\partial_{k}u_{1}\|_{L^{q,\infty}(B(\rho))}.
\ea$$
Thus, we further deduce that
$$\ba
\|\partial_{k}u_{1}\|_{L^{q,\infty}(B(\mu))}\leq& C \|R_{2}R_{k}(\phi w_{3})+R_{3}R_{k}(-\phi w_{2})\|_{L^{q,\infty}(B(\mu))}+C\|H(x)\|_{L^{q,\infty}(B(\mu))}
\\\leq& C  \|R_{2}R_{k}(\phi w_{3})\|_{L^{q,\infty}(B(\rho))}+C\|R_{3}R_{k}(-\phi w_{2})\|_{L^{q,\infty}(B(\rho))}
\\&+C \B(\f{\mu}{\rho}\B)^{3}\|\partial_{k}u_{1}\|_{L^{q,\infty}(B(\rho))}.
\ea$$
The boundedness of Riesz Transform \eqref{brl} in Lorentz spaces gives
$$\ba
\|\partial_{k}u_{1}\|_{L^{q,\infty}(B(\mu))} \leq& C\|   w  \|_{L^{q,\infty}(B(\rho))}
+C \B(\f{\mu}{\rho}\B)^{3}\|\partial_{k}u_{1}\|_{L^{q,\infty}(B(\rho))}.
\ea$$
As a consequence, it comes out
$$\ba
\|\nabla u_{1}\|_{L^{q,\infty}(B(\mu))} \leq& C\|   w  \|_{L^{q,\infty}(B(\rho))}
+C \B(\f{\mu}{\rho}\B)^{3}
\|\nabla u_{1}\|_{L^{q,\infty}(B(\rho))}.
\ea$$
This achieves the proof of the desired estimate.
\end{proof}

\begin{proof}[Proof of \eqref{3inth1.1} and \eqref{4inth1.1} in Theorem \ref{the1.1}]
With Lemma \ref{lemma3.1} in hand, the iterative method and the results \eqref{them1.2} in Theorem \ref{the1.1}  help us finish the proof of \eqref{3inth1.1}   in Theorem.  Since there holds generalized Biot-Savart law \eqref{gbsl},  along the same line, we can complete the rest proof.
\end{proof}

\section{Number of potential singular set in lorentz spaces}
In the same way as \cite{[KRZ]}, one can finish the proof of Theorem \ref{the1.4}. For details we refer the reader to
\cite{[KRZ]}. Next, we follow the path of \cite{[WZ]} to prove Theorem \ref{the1.5}.
\begin{proof}[Proof of Theorem \ref{the1.5}]
From remark \ref{remark1.4}, we know that the weak solution $u$ is a suitable weak solution.
We denote $\mathcal{S}(t_0)={(x_1, t_0),(x_1, t_0),\cdots,(x_M,t_0)}$ and $B_{i}(\varrho)=B(x_{i},\varrho)$ in this section.
We can  choose sufficiently small $r_{0}$  to make sure that $B_i(r)\cap B_j(r)=\emptyset$ for $i\neq j$ for all $0<r\leq r_0$.
It suffices to show that
$$M\leq C \|u\|_{L^{p,\infty}(-1,0;L^{q,l}(\mathbb{R}^{3}))},$$
 where  $\f{2}{p}+\f{3}{q}=1 ( 3<p<\infty \text{  and } p\leq l )$.

To simplify the presentation, we also introduce
$$\ba
&E_{p;q,l;i}(\varrho)=\varrho^{-[(3-q)\f{p}{q}+2]}\|u\|^{p}_{L^{p}
(t_{0}-\varrho^{2},t_{0};L^{q,l}(B_{i}(\varrho)))},\\
&D_{p;q,l;i}(\varrho)=\varrho^{-[(3-2q)\f{p}{q}+2]}\|\Pi\|^{p/2}_{L^{p}
(t_{0} -\varrho^{2},t_{0};L^{q,l}(B_{i}(\varrho)))}.\ea
$$
Lemma \eqref{presure} ensures that
\be\ba
D_{\f{p^{\natural}}{2} ;\f{ q}{2}, \f{ l}{2};i}( \theta r_{0})&\leq C_{1}\theta^{-(\frac{3 p^{\natural}}{q}+2-p)}E_{p^{\natural};q,l;i}(r^*)
+C_{2}\theta^{(2-\frac{4}{p^{\natural}})\f{p^{\natural}}{2}}D_{\f{p^{\natural}}{2}; \f{ q}{2}, \f{ l}{2};i} (r_{0})\\
&= \alpha E_{p^{\natural};q,l;i}(r^*)
+ \beta D_{\f{p^{\natural}}{2} ;\f{ q}{2}, \f{ l}{2};i} (r^*).
\ea\ee
As a consequence, one has
\be\ba
D_{\f{p^{\natural}}{2} ;\f{ q}{2}, \f{ l}{2};i}( \theta^{k}r^*)&\leq \alpha\sum_{i=0}^{k-1}\beta^{k-i-1} E_{p^{\natural};q,l;i}(r^*)
+ \beta^{k} D_{\f{p^{\natural}}{2} ;\f{ q}{2}, \f{ l}{2};i} (r^*).
\ea\ee
Recall that  the elementary inequality
\be\sum_{i=1}^{M}b_{i} ^{\gamma}\leq\left\{\ba
&(\sum_{i=1}^{M}b_{i})^{\gamma},~~~ \gamma\leq1.
\\& M^{1-\gamma}(\sum_{i=1}^{M}b_{i})^{\gamma} , ~~~\gamma<1.
\ea\right.\label{ei}\ee
We take into account the case $1\leq\f{p^{\natural}}{l}$.
With the help of \eqref{ei} and the H\"older inequality in Lorentz spaces, we conclude that, for $p^{\natural}<p$,
$$\ba
&\sum_{i=1}^{M}(\theta^{j}r^*)^{-[-p^{\natural}+\f{3 p^{\natural}}{q}+2]}E_{p^{\natural};q,l;i}(\theta^{j}r_{0})\\
=&\sum_{i=1}^{M}\int^{t_{0}}_{t_{0}-(\theta^j r^*)^{2}}\B(\int_{0}^{\infty}\sigma^{l-1}|\{x\in B_{i}( \theta^j  r^* ):|f|>\sigma\}|^{\f1q}\B)^{\f{p^{\natural}}{l}}dt\\\leq
&\int^{t_{0}}_{t_{0}-(\theta^j r^*)^{2}}\B(\sum_{i=1}^{M}\int_{0}^{\infty}\sigma^{l-1}|\{x\in B_{i}( \theta^j  r^* ):|f|>\sigma\}|^{\f1q}\B)^{\f{p^{\natural}}{l}}dt\\\leq
&\int^{t_{0}}_{t_{0}-(\theta^j r^*)^{2}}\B( \int_{0}^{\infty}\sigma^{l-1}|\{x\in \mathcal{S}:|u|>\sigma\}|^{\f1q}\B)^{\f{p^{\natural}}{l}}dt\\\leq
& (\theta^j r^*)^{2 p^{\natural}(\f{1}{p^{\natural}}-\f{1}{p})}\|u\|^{p^{\natural}}_{L^{p,\infty}(t_{0},t_{0}-(\theta^i r^*)^{2};L^{q,l}(\mathbb{R}^{3}))},
\ea$$
This together with $\f{2}{p}+\f{3}{q}=1$   implies that
$$\ba
&\sum_{i=1}^{M} E_{p;q,l;i}(\theta^{j}r_{0}) \leq C\|u\|^{p^{\natural}}_{L^{p,\infty}(-1,0;L^{q,l}(\mathbb{R}^{3}))},
 \ea$$
 Along exactly the same lines, we find
 $$\ba
&\sum_{i=1}^{M}D_{\f{p^{\natural}}{2} ;\f{ q}{2}, \f{ l}{2};i}( \theta^{k}r^*)\leq C\|\Pi\|^{p^{\natural}/2}_{L^{p/2,\infty}(-1,0;L^{q/2,l/2}(\mathbb{R}^{3}))}.
 \ea$$
We derive from \eqref{the1.2} that
$$\varepsilon<E_{p;q,l;i}(\theta^{k}r_{0})+D_{\f{p^{\natural}}{2} ;\f{ q}{2}, \f{ l}{2};i}( \theta^{k}r^*),$$
from which follows that
$$\ba
M\varepsilon&<\sum_{i=1}^{M}E_{p;q,l;i}(\theta^{k}r_{0})
+\sum_{i=1}^{M}D_{\f{p^{\natural}}{2} ;\f{ q}{2}, \f{ l}{2};i}( \theta^{k}r^*) \\
\leq& C\|u\|^{p^{\natural}}_{L^{p,\infty}(-1,0;L^{q,l}(\mathbb{R}^{3}))}  +\sum_{l=1}^M\beta ^kD_{\f{p^{\natural}}{2} ;\f{ q}{2}, \f{ l}{2};i}(  r^*)
+\sum_{l=1}^M\alpha\sum_{i=0}^{k-1}\beta^{k-i-1}E_{p;q,l;i}(\theta^{i}r_{0}) \\
\leq& C\|u\|^{p^{\natural}}_{L^{p,\infty}(-1,0;L^{q,l}(\mathbb{R}^{3}))} +\beta ^kC\|\Pi\|^{p^{\natural}/2}_{L^{p/2,\infty}(-1,0;L^{q/2,l/2}(\mathbb{R}^{3}))}
\\&+\alpha\sum_{i=0}^{k-1}\beta^{k-i-1}
\|u\|^{p^{\natural}}_{L^{p,\infty}(-1,0;L^{q,l}(\mathbb{R}^{3}))}.
\ea$$
Thanks to $\beta<1$, for sufficiently large $k$, we obtain
$$
M\leq C \varepsilon^{-1}C\|u\|^{p^{\natural}}_{L^{p,\infty}(-1,0;L^{q,l}(\mathbb{R}^{3}))}.
$$
It remains to consider the case $ \f{p^{\natural}}{l}<1$. Indeed, there holds $$\ba
&\sum_{i=1}^{M}(\theta^{j}r^*)^{-[-p^{\natural}+\f{3 p^{\natural}}{q}+2]}E_{p^{\natural};q,l;i}(\theta^{j}r_{0}) \\ \leq
& M^{1-\f{p^{\natural}}{l}}(\theta^j r^*)^{2 p^{\natural}(\f{1}{p^{\natural}}-\f{1}{p})}\|u\|^{p^{\natural}}_{L^{p,\infty}(t_{0},t_{0}-(\theta^i r^*)^{2};L^{q,l}(\mathbb{R}^{3}))},
\ea$$
Arguing as we did above, we get
 $$
M\leq C \varepsilon^{-\f{l}{p^{\natural}}} \|u\|^{l}_{L^{p,\infty}(-1,0;L^{q,l}(\mathbb{R}^{3}))}.
$$
The number of singular points is finite under the condition \eqref{1.31}.
\end{proof}
\section*{Acknowledgement}
%The authors thank the anonymous referee and the associated editor for the invaluable
%comments and suggestions which helped to improve the paper greatly.
The authors would like to express their sincere gratitude to Dr. Daoguo Zhou  for short discussion on Lorentz spaces.
Wang was partially supported by  the National Natural
Science Foundation of China under grant (No. 11971446  and No. 11601492)  and the Youth Core Teachers Foundation of  Zhengzhou University of Light Industry.
Wei was partially supported by the National Natural Science Foundation of China under grant ( No. 11601423, No. 11701450, No. 11701451, No. 11771352, No. 11871057) and Scientific Research Program Funded by Shaanxi Provincial Education Department (Program No. 18JK0763).

%\end{CJK*}
\end{document}